\documentclass[11pt,reqno]{amsart}

\newcommand\version{December 6, 2018}

\usepackage{amsmath,amsfonts,amsthm,amssymb,amsxtra,graphicx}
\usepackage{tikz}

\usepackage{color}

\newtheorem{theorem}{Theorem}[section]
\newtheorem{proposition}[theorem]{Proposition}
\newtheorem{lemma}[theorem]{Lemma}
\newtheorem{corollary}[theorem]{Corollary}

\theoremstyle{definition}

\newtheorem{definition}[theorem]{Definition}

\theoremstyle{remark}

\newtheorem{remark}[theorem]{Remark}

\numberwithin{equation}{section}

\renewcommand{\epsilon}{\varepsilon}

\newcommand{\N}{\mathbb{N}}

\newcommand{\R}{\mathbb{R}}

\DeclareMathOperator{\tr}{Tr}

\begin{document}

\title[Hanner complement --- \version]{Inequalities for $L^p$-norms that 
sharpen the triangle inequality and complement Hanner's Inequality}

\author{Eric A. Carlen}
\address[Eric A. Carlen]{Department of Mathematics, Hill Center,
Rutgers University, 110 Frelinghuysen Road, Piscataway NJ 08854-8019, USA}
\email{carlen@math.rutgers.edu}

\author{Rupert L. Frank}
\address[Rupert L. Frank] {Mathematisches Institut, Ludwig-Maximilans 
Universit\"at M\"unchen, Theresienstr. 39, 80333 M\"unchen, Germany, and 
Mathematics 253-37, Caltech, Pasa\-de\-na, CA 91125, USA}
\email{rlfrank@caltech.edu}

\author{Paata Ivanisvili}
\address[Paata Ivanisvili] {Department of Mathematics, 
University of California, Irvine, CA 92617, USA}
\email{pivanisv@uci.edu}

\author{Elliott H. Lieb}
\address[Elliott H. Lieb]{Departments of Mathematics and Physics,
Princeton University, Princeton, NJ 08544,
USA}
\email{lieb@princeton.edu}

\thanks{\copyright\, 2018 by the authors. This paper may be reproduced, in
its entirety, for non-commercial purposes.\\
Work partially supported by NSF grants DMS--1501007 (E.A.C.), DMS--1363432 (R.L.F.), PHY--1265118 (E.H.L.)}

\begin{abstract}
In 2006 Carbery raised a question about an improvement on the na\"ive norm inequality
$\|f+g\|_p^p \leq 2^{p-1}(\|f\|_p^p + \|g\|_p^p)$
for
two functions in $L^p$ of any measure space. When $f=g$ this is an 
equality, but when the supports of  $f$ and $g$ are disjoint the factor 
$2^{p-1}$ is not needed. Carbery's question concerns a proposed interpolation between the two 
situations for $p>2$. The interpolation parameter measuring the overlap 
is $\|fg\|_{p/2}$. We prove an inequality of this type that is
stronger than the one Carbery proposed. Moreover, our stronger 
inequalities are valid for {\it all } $p$.

\end{abstract}

\maketitle
\centerline{\version}

\section{Introduction and main theorem}

Since $|z|^p$ is a convex function of $z$ for $p\geq1$, for any measure space, the $L^p$   unit ball, 
$
B_p := \{ f \ :\ \int |f|^p \leq 1 \}
$,
is convex. 
One way to express this is with Minkowski's triangle inequality $\|f+g\|_p 
\leq \|f\|_p + \|g\|_p$. Another is the inequality
\begin{equation}\label{naive}
\|f+g\|_p^p \leq 2^{p-1} \left(\|f\|_p^p + \|g\|_p^p\right),
\end{equation}
valid for any functions $f$ and $g$ on any measure space. There is equality if and only if
$f=g$ and, in Theorem \ref{main},  we improve \eqref{naive} substantially   when $f$ and $g$ are far from equal.

In 2006 Carbery proposed \cite{C} several plausible refinements of \eqref{naive} for 
$p\geq2$, of which the strongest was
\begin{equation} \label{carb}
\int \left|f+g \right|^p \leq \left(1+ \frac{\Vert fg \Vert_{p/2}} {\Vert 
f\Vert_p \Vert g\Vert_p} \right)^{p-1}\int\left(|f|^p +|g|^p\right).
\end{equation}

There is equality in \eqref{carb} both when 
$f=g$ and when $fg=0$.  Thus, \eqref{carb}, if true, can be viewed as a refinement of \eqref{naive} in which there is equality not only when $f=g$ but also when $fg=0$.

The ratio ${\displaystyle \Gamma = \frac{\Vert fg \Vert_{p/2}} {\Vert 
f\Vert_p \Vert g\Vert_p} }$ varies between 0 and 1 and, therefore, the factor of
$(1+\Gamma)^{p-1} $ varies between 1 and $2^{p-1} $, interpolating between the two cases of equality in \eqref{carb}.

We  propose and prove a  strengthening of \eqref{carb} in which 
$\Gamma $ is replaced by the quantity 
\begin{equation} \label{quant}
\widetilde\Gamma:=
\Vert fg \Vert_{p/2} 
\left(\frac
{\Vert f\Vert_p^p + \Vert g\Vert_p^p}{2}\right)^{-2/p}  \  ,
\end{equation}
which is smaller by virtue of the arithmetic-geometric mean inequality.

Our improved inequalities are not restricted to $p>2$, but are valid for 
all $p \in \R$, as stated in Theorem \ref{main}. There we write 
$$\|f\|_p := \left(\int |f|^p \right)^{1/p} \quad \text{for {\bf all}} \ p\neq 0.$$
We note that inequality \eqref{carb} involves
three kinds of 
quantities on the right side ($\Vert fg \Vert_{p/2},\ \Vert f\Vert_p^p + 
\Vert g\Vert_p^p $ and $\Vert 
f\Vert_p \Vert g\Vert_p$), while our inequality involves only two ($\Vert 
fg \Vert_{p/2}$ and $\Vert f\Vert_p^p + 
\Vert g\Vert_p^p $), a simplification that is essential for our proof.

\begin{theorem}[Main Theorem] \label{main}
 For all $ p \in
  (0, 1] \cup [2, \infty )$ and functions $f$ and $g$ on any measure 
space, 
 \begin{equation} \label{maineq}\boxed{
\int \left|f+g\right|^p \leq \left(1+ \frac{2^{2/p}\Vert 
fg \Vert_{p/2}} 
{\left(\Vert f\Vert_p^p + \Vert g\Vert_p^p\right)^{2/p} }
\right)^{p-1} 
\int\left(\, |f|^p +|g|^p\, \right).}
\end{equation}
The inequality reverses if  $ p\in(-\infty,0)\cup(1,2)$, where, for $p\in(1,2)$,  it is assumed that $f$ and $g$ are positive almost everywhere.    \\
For $p>2$,   \ (resp. for $p\in (0,1)$ )
the inequality is false if \, $\widetilde\Gamma $
is raised to any power $q>1$, (resp. $q<1$).\\
For $p<0$,   \ (resp. for $p\in [1,2]$ )
the reversed inequality is false if \, 
$\widetilde\Gamma$
is raised to any power $q>1$, (resp. for $q<1$). 

For $p>0$, $p\neq 1,2$, $\|f\|_p,\|g\|_p < \infty$, there is equality in \eqref{maineq} if and only if $f$ and $g$ have disjoint supports, up to a null set, or 
are equal almost everywhere.
For $p<0$, $\|f\|_p,\|g\|_p < \infty$, there is equality in \eqref{maineq} if and only if $f$ and $g$  are equal almost everywhere.

\end{theorem}

We note that in proving the theorem, we may always assume that $f$ and $g$ are non-negative. In fact, the right side of \eqref{maineq} only depends on $|f|$ and $|g|$ and the left side does not decrease for $p\geq 0$ and does not increase for $p<0$ if $f$ and $g$ are replaced by $|f|$ and $|g|$. The latter follows since $|f+g| \leq |f| + |g|$ implies $|f+g|^p \leq (|f|+|g|)^p$ for $p>0$ and  $|f+g|^p \geq (|f|+|g|)^p$ for $p<0$.

Carbery proved that his proposed inequality is valid  when $f$ and $g$ are characteristic 
functions.  Our theorem can also be easily proved in this special case.

Theorem \ref{main} may be viewed as a refinement of Minkowski's inequality.  
Since 
\eqref{naive}, like Minkowski's inequality, is a direct expression of the 
convexity of $B_p$, it is equivalent to Minkowski's  inequality.  
We recall the simple argument: For any unit vectors $u,v\in L^p$, \eqref{naive} says that $\|(u+v)/2\|_p \leq 1$, and then by continuity, 
$\|\lambda u + (1-\lambda) v\|_p \leq 1$ for all $\lambda\in (0,1)$. 
Suppose $0 < \|f\|_p,\|g\|_p < \infty$, and define $\lambda = 
\|f\|_p/(\|f_p\|+ \|g\|_p)$, $u = \|f\|_p^{-1}f$, and $v = \|g\|_p^{-1}g$. 
Then
$$\| f +g\|_p^p = (\|f\|_p + \|g\|_p)^p \|\lambda u + (1-\lambda)  v\|_p^p \leq  (\|f\|_p + \|g\|_p)^p\ , $$
which is Minkowski's inequality.

When $p=1$ and $f,g \geq 0$, \eqref{naive} is an identity; otherwise when $p>1$,
there is equality in \eqref{naive}  if and only if   
$f=g$.   When the supports of  $f$ and $g$ are disjoint, however, 
\eqref{naive} is far from an equality and the factor 
$2^{p-1}$ is not needed.  There is equality in Minkowski's inequality whenever $f$ is a multiple of $g$ or {\em vice-versa}. Hence  although \eqref{naive} is equivalent to Minkowski's inequality, it  becomes an equality in fewer circumstances.

There is another well-known refinement of Minkowski's inequality for $1 < p < \infty$, namely {\em Hanner's inequality}, \cite{H,BCL,LL} which gives the exact modulus of convexity of $B_p$, the unit ball in $L^p$. For $p\geq 2$, and unit vectors $u$ and $v$, Hanner's inequality
says that 
\begin{equation}\label{clark}
\left\Vert \frac{u+v}{2}\right\Vert_p^p +  \left\Vert 
\frac{u-v}{2}\right\Vert_p^p\leq  1,
\end{equation}
which is also a consequence of one of Clarkson's inequalities \cite{Cl}.  When $u$ and $v$ have disjoint supports, $\|u+v\|_p^p = 
\|u-v\|_p^p =2$, and then the left hand side is $2^{2-p}$, so that for unit 
vectors $u$ and $v$, the condition   $uv= 0$, which yields equality in  
the inequality of Theorem~\ref{main}, does not yield equality in Hanner's 
inequality. On the other hand, while one can derive a bound on the modulus 
of convexity in $L^p$ from \eqref{maineq}, one does not obtain the sharp 
exact result provided by Hanner's inequality.
Both inequalities express a quantitative strict convexity property of $B_p$, 
but neither implies the other; they provide complimentary information, with 
the information provided by Theorem~\ref{main} being especially strong when 
$f$ and $g$ have small overlap as measured by $\|fg\|_{p/2}$.

Our proof of Theorem~\ref{main} consists of three parts:

\medskip

\quad Part A:  We show how to reduce the inequality to a simpler one 
involving only one function, namely $\alpha := f/(f+g)$ for $f,g>0$, which takes values in $[0,1]$, and a reference measure that is a probability measure.  This exploits 
the fact that the only important quantity is the ratio of $f$ to 
$g$. This part is very easy.

\medskip

\quad Part B: In the second part, which  is more difficult than Part A,  we 
show that Theorem~\ref{main} is true if it
is true when the  function $\alpha$ is constant. (This is the same as 
saying $f$ 
and $g$ are proportional to each other.) When $\alpha := f/(f+g)$ is constant 
and the reference measure is a probability measure, \eqref{maineq} yields 
the inequality for \emph{numbers} $\alpha\in [0,1]$ and $p\in [0,1] \cup [2,\infty]$,
\begin{equation}\label{constantA}
  1 \leq \left( 1+ \left (\frac{2 \alpha^{p/2}(1-\alpha)^{p/2}}{\alpha^p +(1-\alpha)^p} 
\right)^{2/p} \right)^{p-1} \left( \alpha^p +(1-\alpha)^p \right) \ .
\end{equation}
with the reverse inequality for $p\notin [0,1] \cup [2,\infty]$. 

\begin{remark} The quantity ${\displaystyle R := \frac{2 \alpha^{p/2}(1-\alpha)^{p/2}}{\alpha^p +(1-\alpha)^p}}$ lies in $[0,1]$ for all $\alpha$ and $p$.   Therefore, $R^q$ decreases as $q$ increases.  Thus for $p\geq 2$, the inequality 
\begin{equation}\label{constantB}
  1 \leq \left( 1+ \left (\frac{2 \alpha^{p/2}(1-\alpha)^{p/2}}{\alpha^p 
+(1-\alpha)^p} 
\right)^{q} \right)^{p-1} \left( \alpha^p +(1-\alpha)^p \right) \ ,
\end{equation}  strengthens as  $q$ increases,  and for 
$p\in [0,1]$, it strengthens  as  $q$ decreases. Likewise, for $p\in [1,2]$ the reverse of \eqref{constantB}  is stronger for smaller $q$, and for $p < 0$, it is stronger for larger $q$. 
\end{remark}

\quad Part C: 
With Parts A and B complete,  the proof reduces to a 
seemingly elementary  inequality, parametrized by $p$, for a 
number $\alpha\in [0,1]$.  The proof of this is Part C. While the validity of \eqref{constantA} appears to be a 
consequence of Theorem \ref{main},  one can also view Theorem \ref{main} as a
consequence of \eqref{constantA}.

\begin{theorem}\label{meansthm}  For all numbers $\alpha\in [0,1]$,
inequality \eqref{constantA} is valid for 
all $p\in (0,1]\cup [2,\infty)$,   and the  reverse 
inequality is valid for all
$p\in(-\infty,0)\cup(1,2)$.

For $p>2$,   \ (resp. for $p\in (0,1)$ )
inequality~\eqref{constantB} is false if $q>2/p$, (resp. for $q<2/p$).

For $p<0$,   \ (resp. for $p\in [1,2]$ )
the reverse inequality is false if 
$q>2/p$, (resp. for $q<2/p$).  

For $p>0$, $p\neq 1,2$, there is equality if and only if $\alpha\in \{0,1/2,1\}$. For $p<0$, there is equality if and only if $\alpha =1/2$. 
\end{theorem}

\subsection{Restatement of Theorem \ref{meansthm} in terms of means}

Inequality \eqref{constantA} can be restated in terms of $q$th power 
means \cite{HLP}:  For $x,y>0$, define
$$
M_q(x.y) = ((x^q + y^q)/2)^{1/q}
\quad\text{if}\ q\in\R\setminus\{0\}
\qquad\text{and}\qquad
M_0(x,y)=\sqrt{xy} \,.
$$
Note that $M_0(x,y)$ is the \emph{geometric mean} of $x$ and $y$ and $M_{-1}(x,y)$ is their \emph{harmonic mean}.

\begin{corollary}  For all $x,y> 0$, and all  $p\in [0,1] \cup [2,\infty]$
\begin{equation}\label{qmeans}
M_1^p(x,y)  \leq \left(\frac{ M_p(x,y) + 
M_{-p}(x,y)}{2}\right)^{p-1}M_p(x,y)\ ,
\end{equation}
while the reverse inequality is valid for all $p\in(-\infty,0)\cup(1,2)$.
\end{corollary}

\begin{proof} A simple calculation shows that for all $p > 0$, 
${\displaystyle\frac{M_{-p}(x,y)}{M_p(x,y)} = \frac{2^{2/p}xy}{(x^p + y^p)^{2/p}}}$.
Thus, taking $x=\alpha$ and $y = 1-\alpha$, the inequality \eqref{constantA} can be written as
$$
\frac12 \leq  \left( 1 +  
\frac{M_{-p}(\alpha,1-\alpha)}{M_p(\alpha,1-\alpha)} \right)^{p-1} 
M_p^p(\alpha,1-\alpha)\ ,$$
Then by homogeneity and the fact that $M_1(\alpha,1-\alpha) = 1/2$, \eqref{constantA} is equivalent to \eqref{qmeans}
\end{proof}

The following way to write our inequality sharpens and 
complements  the arithmetic-geometric mean inequality  for any two numbers 
$x,\,  y >0$, provided one has information on $M_p(x,y)$. 

\begin{corollary}[Improved and complemented AGM inequality]  For all $x,y > 
0$, and all $p > 2$, 
\begin{equation} 
 1 - \left(\frac{A}{M_p} \right)^{p'}  \, \geq  \,  
\frac12\left( 1 -\left(\frac{G}{M_p}\right)^2 \right)    \geq
\frac12\left(1- \left(\frac{G}{M_{p'}}\right)^2 \right)  \geq  1-  
\left(\frac{A}{M_{p'}} \right)^{p} 
 \end{equation} 
where $p' = p/(p-1)$, $A=(x+y)/2$ and $G = \sqrt{xy}$.
\end{corollary}

\begin{remark} Since $p,p' \geq 1$, all of the quantities being compared in these inequalities are non-negative. 
\end{remark}

Despite the classical appearance of \eqref{qmeans}, we have not been able to 
find it in the literature, most of which concerns inequalities for means 
$M_q(x_1,\dots, x_n) = (\frac1n\sum_{j=1}^n x_j^p)^{1/p}$ of an $n$-tuple of 
non-negative numbers, often with more general weights.  The obvious 
generalization of \eqref{qmeans} from two
to three non-negative numbers $x$, $y$, and $z$ is false as one sees by 
taking $z=0$: Then there is no help from $M_{-p}(x,y,z)$ on the right. A 
valid generalization to more variables probably involves means over 
$M_{-p}(x_j,x_k)$ for the various pairs. In any case, as far as we know, 
\eqref{qmeans} is new.

A truly remarkable feature of the inequality \eqref{qmeans} 
is that it is {\em surprisingly close to equality uniformly in the arguments}. To see 
this, let $f(\alpha,p)$ denote the right hand side of \eqref{constantA}. Contour plots of this function for various ranges of $p$ are shown in Figs. 1, 2 and 3 below.


\medskip

\centerline{ 
\includegraphics[width=1.8 in]{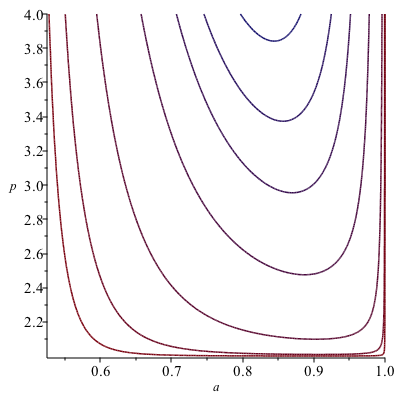} \quad
\includegraphics[width=1.8 in]{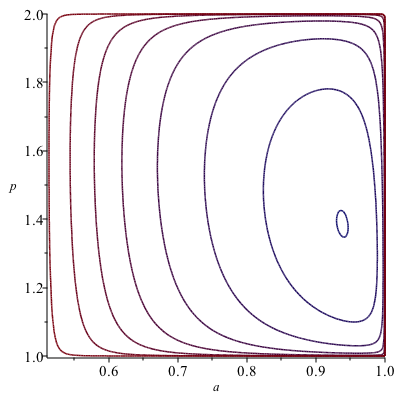}\quad
\includegraphics[width=1.8 in]{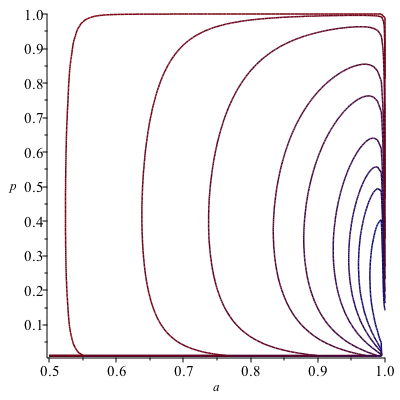} }
\centerline{\footnotesize Fig.~1\hskip 1.6 true in  Fig.~2  \hskip 1.6 true in Fig.~3}\nobreak


Fig. 1 is a contour plot of this function in $[1/2,1]\times [2,4]$. 
The contours shown in Fig.~1 range from $1.00001$ to $1.018$. Note that the function $f$ is identically $1$ along three sides of plot: $\alpha =1/2,1$, and $p=2$. The maximum value for $2 \leq p \leq 4$, near $1.018$, occurs towards the middle of the segment at $p=4$.

Fig.~2  is a contour plot of $f$   on  $[1/2,1]\times [1,2]$. The contours range from $0.9961$ (the small closed contour) to $0.99999999$ (close to the boundary). 
\emph{ Amazingly, the function in \eqref{constantA} is quite close  -- within two percent --
to the constant 1 over the range $p\geq 1$ and $\alpha\in [0,1]$.}  Moreover, the ``landscape'' is quite flat: The gradient has a small norm over the whole domain. 


Fig.~3 is a contour plot of $f$ in the domain  
$[0,1/2]\times[0,1]$. The contours in Fig.~3 range from  $1.0000001$ to $1.06$. Higher values are to the right. 
For $p$ in this range, the maximum is not so large -- about $1.06$ -- but the landscape gets very ``steep'' near $\alpha =1$ and $p =0$.   The proof of the inequality is especially delicate in this case.

For $p < 0$, there is equality only at $\alpha =1/2$, and the inequality is not so uniformly close to an identity. The contour plot is less informative, and hence is not recorded here.   This is the case in which the inequality is easiest to prove.

It is possible to give a simple direct proof of the inequality for certain integer values of $p$, as we discuss in Section 5. We also give a simple proof that for $p>2$ and for $p<0$, validity of the inequality at $p$ implies validity of the inequality at $2p$, and we briefly discuss an application of this to the problem in which functions are replaced by operators and  integrals are replaced by traces.

\begin{remark}\label{alter}
We close the introduction by briefly discussing one other way to write the inequality \eqref{constantA}. Introduce a new variable $s\in(0,1)$ through
$$
\alpha = \frac{1+\sqrt{s}}{2}
$$
Rewriting  \eqref{constantA}, and taking   the $\frac{1}{p-1}$ root of both sides, we may rearrange terms to obtain. 
\begin{equation}\label{cons21}
2   \leq \eta^{\frac{1}{p-1}}(s)\left( 1 + \frac{1-s}{\eta^{\frac2p}(s)}\right) = \eta^{\frac{1}{p-1}}(s) + (1-s)  \eta^{\frac{2-p}{p(p-1)}}(s) 
\end{equation}
for $0 \leq s \leq 1$, where
\begin{equation}\label{cons22}
\eta(s) := \frac{(1+\sqrt{s})^p + (1-\sqrt{s})^p}{2}\ .
\end{equation}
Taking the $\frac{1}{p-1}$ eliminate the change of direction in the inequality at $p=1$, and it now take on a non-trivial form at $p=1$:
Define 
\begin{equation}\label{fpdef}
f_p(s) :=  \eta^{\frac{1}{p-1}}(s) + (1-s)  \eta^{\frac{2-p}{p(p-1)}}(s) -2\ ,
\end{equation}
for  $p\neq1$, and one easily computes the  limit at $p=1$:
$$
f_1(s) := (2-s)(1-\sqrt{s})^{\frac{1-\sqrt{s}
}{2}}(1+\sqrt{s})^{\frac{1+\sqrt{s}
}{2}} -2\ .
$$
 Theorem~\ref{meansthm} is equivalent to the assertion that for all $s\in (0,1)$
\begin{equation}\label{toshow}
f_p(s) \geq 0 \ {\rm for} \ p\in (-\infty, 0) \cup (2,\infty)  \quad{\rm and}\quad   f_p(s) \leq  0 \ {\rm for} \ p\in (0,2) \ .
\end{equation}
In this form, the inequality is easy to check for some values of $p$. For example, for $p=-1$, $\eta(s) = \frac{1}{1-s}$ and $f_{-1}(s) = (1-s)^{1/2} + (1-s)^{-1/2} -2$. 
which is clearly positive. One can give simple proofs of \eqref{toshow} for other integer values of $p$, e.g., $p=3$ and $p=4$ along these line, but this change of variables is {\em not} what we use to prove the general inequality. It is, however, convenient for checking optimality of of the power $2/p$ in  \eqref{constantA}.
\end{remark}


\section{Part A. \  Reduction from two functions to one }

While Theorem \ref{main}
involves  two functions $f$ and $g$ one can use the arbitrariness of the 
measure to reduce the question to a \emph{single} function defined on a \emph{probability space} (that is, $\int 1=1$). We have already observed that 
it suffices to prove the inequality in the case where $f$ and $g$ are both 
non-negative. For non-negative functions $f$ and $g$,  
set
$$
\alpha = f/(f+g) \,,
\qquad
1-\alpha = g/(f+g) \,.
$$
Replacing the underlying measure $dx$ by the new measure $(f+g)^p 
\,dx/\|f+g\|_p^p$ we see that it suffices to prove the following inequality 
for $p\in [0,1] \cup [2,\infty]$, and also to prove the reverse inequalities
for $p \notin [0,1] \cup [2,\infty]$\ :
\begin{equation}\label{uno} 
1 \leq \left( 1+ 
\frac{2^{2/p}\, \|\alpha(1-\alpha)\|_{p/2}}
{ \left( \, \| \alpha\|_p^p + 
\|1-\alpha\|_p^p \, \right)^{2/p} } \right)^{p-1} \left(\,  \|\alpha\|_p^p 
+ \|1-\alpha\|_p^p \, \right)
\end{equation}
for a single function $0\leq\alpha\leq 1$ on a \emph{probability space}, i.e., $\int 1 =1$.


\section{Part B.\ Reduction to a constant function}

In this section we prove the following. 

\begin{proposition}\label{conred}
If $p\in[0,1]\cup[2,\infty)$, then inequality \eqref{uno} is true for 
all functions $\alpha$ (which is equivalent to \eqref{carb} for all $f,g$) 
if and only if it is true for all constant functions, that is, for all numbers $\alpha \in [0,1]$,
\begin{equation}\label{constant}
  1 \leq \left( 1+ \left (\frac{2 \alpha^{p/2}(1-\alpha)^{p/2}}{\alpha^p +(1-\alpha)^p} 
\right)^{2/p} \right)^{p-1} \left( \alpha^p +(1-\alpha)^p \right).
\end{equation}
If $p\notin[0,1]\cup[2,\infty)$, then the reverse of inequality \eqref{uno} is true for all functions $\alpha$ (which is equivalent to the reverse of \eqref{carb} for all $f,g$) if and only if it is true for all constant functions, that is, for all numbers $\alpha \in [0,1]$, the reverse of \eqref{constant} holds.

Moreover, for  $p\neq 0, 1,2$, there is equality in \eqref{uno} if and only if $\max\{\alpha(x),1-\alpha(x)\}$ is constant almost everywhere.
\end{proposition}

To prove this Proposition we need a definition and a lemma. 

\begin{definition}
Fix $p\in \R$ and for $0\leq a \leq 1$, let $h(a) := a^{p/2}(1-a)^{p/2}$ 
and let $b(a) :=
a^p 
+ (1-a)^p$. Clearly, $b$ determines the unordered pair $a$ and $1-a$ and, 
therefore, $b$ determines $h$.  Thus, we can consider the function $b 
\mapsto H(b) := h( a^{-1}(b) )$ (in which  the dependence on
$p$ is suppressed in the notation). 
\end{definition}

\begin{lemma}[convex/concave $H$]\label{con}
The function $b \mapsto H(b) $ is strictly convex when $p\in (2,\infty)$  and 
strictly concave when $p\in (-\infty, 2)$, $p \neq 0,1$. 
\end{lemma}

\begin{proof}
To prove this lemma we use the chain rule to compute the second derivative 
of $H$.
As a first step we define a useful reparametrization as follows: $e^{2x} := 
a/(1-a)$. A quick computation shows that
$h= (2 \cosh x)^{-p}$ and  $b= 2 \cosh (px) (2 \cosh x)^{-p}$. Thus,
$h= b/(2\cosh (px))$. By symmetry, we can restrict our attention to the 
half-line $x\geq 0$.

We now compute the first two derivatives:

\begin{eqnarray}
db/dx &= &  2^{1-p} p \frac{ \sinh ((p-1)x ) } { (\cosh x)^{p+1} 
}\label{first}  \\
dh/dx & = & -p \frac{\tanh x }{(\, 2 \cosh x \, )^p }\\
(dH/db)(x) & = & \frac{dh/dx}{db/dx} = - \frac{\sinh x}{2 \sinh ((p-1)x) }\\
(d/dx)(dH/db)(x) & = & \cosh(x) \frac{(p-1)\tanh x -
\tanh((p-1) x ) 
}{2\sinh((p-1)x)\ \tanh(
(p-1)x)} \label{second} \\
(d^2H/db^2)(x) & = & \frac{(d/dx)(dH/db)(x)}{db/dx} \label{third}
\end{eqnarray}
Our goal is to show that \eqref{third} has the correct sign (depending on $p$) for 
all $x\geq0$. 

Clearly, the quantity \eqref{first} is nonpositive for $p\in (0,1]$ and 
nonnegative elsewhere.  
We claim that the quantity \eqref{second} is nonpositive for $p\in (-\infty,0]\cup [1,2]$ and nonnegative elsewhere. In fact, the denominator is always 
positive. For the numerator we write $t=p-1$ and use the fact that for all 
$x> 0$, $t\tanh(x) - \tanh (tx) > 
0 $
for $t > 1 $ and for $-1< t < 0$, while the 
inequality reverses, and is strict for other values of $t$ except $t\in \{-1,0,1\}$.

To see this, fix $x>0$, and define 
$f(t) := t\tanh(x) - \tanh (tx)$. Evidently $f(t) = 0$ for $t=-1,0,1$. Then, since $f''(t) = 2x^2\sinh(tx)/\cosh^3(tx)$, $f''(t) > 0$ for $t>0$, and $f''(t) < 0$ for $t<0$. It follows that $f(t) > 0$ for $-1 < t < 0$ and $t>1$, while $f(t) < 0$ for $0 < t < 1$ and $t < -1$.

According to \eqref{third} the  products of the signs of \eqref{first} and  \eqref{second} yield the 
strict convexity/concavity properties of $H(b) $ shown in rows $2$ to $4$  of the table below. 
\end{proof}

\begin{center}
\begin{tikzpicture}
\draw [ultra thick] (0,0) -- (0,6);
\draw [ultra thick] (2,0) -- (2,6);
\draw [ultra thick] (4,0) -- (4,6);
\draw [ultra thick] (6,0) -- (6,6);
\draw [ultra thick] (8,0) -- (8,6);
\draw [ultra thick] (10,0) -- (10,6);

\draw [ultra thick] (0,0) -- (10,0);
\draw [ultra thick] (0,1) -- (10,1);
\draw [ultra thick] (0,2) -- (10,2);
\draw [ultra thick] (0,3) -- (10,3);
\draw [ultra thick] (0,4) -- (10,4);
\draw [ultra thick] (0,5) -- (10,5);
\draw [ultra thick] (0,6) -- (10,6);

\node [right] at (2.4,5.5,0) {{\footnotesize$p<0$}};
\node [right] at (4.15,5.5,0) {{\footnotesize$0 <p<1$}};
\node [right] at (6.15,5.5,0) {{\footnotesize$1 <p<2$}};
\node [right] at (8.4,5.5,0) {{\footnotesize$p>2$}};

\node [right] at (0.6,4.5,0) {{\footnotesize $\tfrac{db}{dx}$}};
\node [right] at (2.6,4.5,0) {{\footnotesize $\geq 0$}};
\node [right] at (4.6,4.5,0) {{\footnotesize $\leq 0$}};
\node [right] at (6.6,4.5,0) {{\footnotesize $\geq 0$}};
\node [right] at (8.6,4.5,0) {{\footnotesize $\geq 0$}};

\node [right] at (0.2,3.5,0) {{\footnotesize $\tfrac{d}{dx}\left(\tfrac{dH}{db}\right)$}};
\node [right] at (2.6,3.5,0) {{\footnotesize $\leq 0$}};
\node [right] at (4.6,3.5,0) {{\footnotesize $\geq 0$}};
\node [right] at (6.6,3.5,0) {{\footnotesize $\leq 0$}};
\node [right] at (8.6,3.5,0) {{\footnotesize $\geq 0$}};

\node [right] at (0.5,2.5,0) {{\footnotesize $H(b)$}};
\node [right] at (2.2,2.5,0) {{\footnotesize concave}};
\node [right] at (4.2,2.5,0) {{\footnotesize concave}};
\node [right] at (6.2,2.5,0) {{\footnotesize concave}};
\node [right] at (8.3,2.5,0) {{\footnotesize convex}};

\node [right] at (0.2,1.5,0) {{\footnotesize $p(p-1)$}};
\node [right] at (2.6,1.5,0) {{\footnotesize $\geq 0$}};
\node [right] at (4.6,1.5,0) {{\footnotesize $\leq 0$}};
\node [right] at (6.6,1.5,0) {{\footnotesize $\geq 0$}};
\node [right] at (8.6,1.5,0) {{\footnotesize $\geq 0$}};

\node [right] at (0.1,0.5,0) {{\footnotesize Direction}};
\node [right] at (2.7,0.5,0) {{\footnotesize $\geq$}};
\node [right] at (4.7,0.5,0) {{\footnotesize $\leq$}};
\node [right] at (6.7,0.5,0) {{\footnotesize $\geq$}};
\node [right] at (8.7,0.5,0) {{\footnotesize $\leq$}};
\end{tikzpicture}
\end{center}

\centerline{Fig. 6: Table of signs determining the direction of the main
inequality \eqref{maineq}.}

\begin{lemma}[constant $b(\alpha(x))$]\label{constb} For $f,g \geq 0$, define $\alpha(x) = f(x)/(f(x) + g(x))$.
Then with $b(\alpha(x))$ is almost everywhere constant if $\max\{\alpha(x),1-\alpha(x)\}$ is constant almost everywhere, which is true if and only if  only if either $f$ and $g$ have essentially disjoint support, or else $\tilde{f}(x) := \max\{f(x),g(x)\}$ and $\tilde{g}(x) := 
\min\{f(x),g(x)\}$ are proportional.
\end{lemma}

\begin{proof} Let  $F = \{x:\ f(x) > 0\}$ and $G = \{x:\ g(x) > 0\}$.  Then $\alpha(x) = 1$ on $F\backslash G$, and
$1- \alpha(x) = 1$ on $G\backslash F$. If the measure of $F\cap G$ is zero, $b(\alpha(x)) =1$ almost everywhere with respect to $(f(x)+g(x)){\rm d}x$. Conversely, since $b(a) =1$ if and only if $a\in\{0,1\}$, if 
$b(\alpha(x)) =1$ almost everywhere, then almost everywhere $\alpha(x)\in\{0,1\}$, which means that $f$ and $g$ have essentially disjoint supports.

For $b\in [2^{1-p},1)$, there is a unique $a\in [1/2,1)$ such that $b(a) = b$. Therefore, if $b(\alpha(x))=b \in [2^{1-p},1)$, there is a unique $a\in [1/2,1)$ such that $\alpha(x) \in \{a,1-a\}$ almost everywhere, and this is
the case if and only if $\max\{\alpha(x),1-\alpha(x)\} = \tilde{f}(x)/(\tilde{f}(x) + \tilde{g}(x)) =a$ almost everywhere.
\end{proof}

\begin{proof}[Proof of Proposition \ref{conred}]

Consider the ratio in \eqref{uno}.  The numerator is the integral
$ \int H(b(\alpha(x)))$. By Jensen's 
inequality (recalling that $\int 1 = 1$) and the convexity/concavity of $H$ in Lemma \ref{con},  this 
integral is bounded from below
by $H(B) $ in the convex case and from above in the concave case, where
\begin{equation}
 B := \int (\alpha^p + (1-\alpha)^p )\ .
\end{equation}

That is, 
\begin{equation}\label{jenin}
 \frac{\int H(b(y))}{B} \geq \frac{ H(B) }{B}
\end{equation}
for $p \geq 2$, while the reverse is true for $p\leq 2$. 
 Moreover, by the strict convexity/concavity of $H(b)$, the inequality in \eqref{jenin} is strict unless $b(\alpha(x))$ is constant when $p\notin\{0,1,2\}$.  By the first part of Lemma~\ref{constb}, $b(\alpha(x))$ is constant if and only if $\max\{\alpha(x),1-\alpha(x)\}$ is a constant, necessarily belonging to $[1/2,1]$.
Then, taking into account the signs of $2/p$ and $p-1$ in the various 
ranges, 
$$
\left( 1+ \left (\frac{2\int H(b(y)) {\rm d}y}{B} \right)^{2/p} \right)^{p-1}  \geq \left( 1+ \left (\frac{2H(B)}{B} \right)^{2/p} \right)^{p-1}
$$
for $p\in   (0,1] \cup [2,\infty)$, with the reverse in 
equality for $p\in (-\infty,0) \cup [1,2]$.    The last two rows in Fig.~6  summarize the interaction of the convexity/concavity properties of $H(b)$ and the signs of the exponents $p/2$ and $p-1$ in the direction of the inequality in \eqref{duo} for the different ranges of $p$, 
and taking into account the cases of equality discussed above, this yields the result as stated. 

Thus, it suffices for us to prove
\begin{equation}\label{duo}
  1 \leq \left( 1+ \left (\frac{2H(B)}{B} \right)^{2/p} \right)^{p-1} B
\end{equation}
for $p\in [0,1] \cup [2,\infty]$ and the reverse inequality for $p\notin [0,1] \cup [2,\infty]$. 
  We do not know what the number $B$ is, but that does not matter. In each case the range of $b$ is an interval and, therefore, the average value $B$ lies in this same interval. Consequently, whatever $B$ might be, there is a number $\alpha$ such that 
  $B= \alpha^p +(1-\alpha)^p$. (Note that it is not claimed that this number $\alpha$ is related in any particular way to the function $\alpha (x)$.)
\end{proof}

\begin{proof}[Proof of Theorem~\ref{main}]  This is immediate from Proposition~\ref{conred} and Theorem~\ref{meansthm}.
\end{proof}

\bigskip 



\section{Part C.\ Proof of  Theorem~\ref{meansthm}}
\def\a{\alpha}

\subsection{Proof of the inequality}

First we prove the inequality
\begin{align}\label{utoloba}
 (\a^{p}+(1-\a)^{p})\left( 1+\left(\frac{2\a^{p/2}(1-\a)^{p/2}}{\a^{p}+(1-\a)^{p}}\right)^{2/p}\right)^{p-1}\geq 1
\qquad\text{for all}\ \alpha\in(0,1)
\end{align}
if $p \in [0,1]\cup [2,\infty)$, and the reverse inequality if $p \in (-\infty, 0] \cup [1,2]$.

\vskip0.5cm
For $p>0$, there is evidently equality for $\a\in \{0, \tfrac12,1\}$, and for $p<0$, there is equality for $\alpha=1/2$. Thus for the proof of \eqref{utoloba} it suffices to consider $\alpha\in(1/2,1)$ for $p>0$, and $\alpha \in (0,1/2)$ if $p<0$,   and it is convenient to change variables
$$
t:=\left(\frac{1-\a}{\a}\right)^{p} \in (0,1]
\qquad\text{and}\qquad
c:=1/p \,.
$$
Moreover, for fixed $c$ we introduce the function
$$
f(t):=-\frac{1}{c} \ln (1+t^{c}) + \ln(1+t)+ \frac{1-c}{c}\ln\left(1+ \left(\frac{4t}{(t+1)^{2}}\right)^{c}\right).
$$
By taking logarithms we see that the claimed inequality \eqref{utoloba} is equivalent to
$$
f(t) \geq 0
\qquad\text{for}\ t\in(0,1)
$$
if $p\in[0,1]\cup[2,\infty)$ (that is, $c\in(0,1/2]\cup[1,\infty)$), and the reverse inequality in \eqref{utoloba} is equivalent to the reverse inequality if $p \in (-\infty, 0] \cup [1,2]$ (that is, $c\in(-\infty,0]\cup[1/2,1]$). We shall show that for $c>0$ the derivative $f'$ has a unique sign change in $(0,1)$ and it changes sign from $+$ to $-$ if $c\in(0,1/2)\cup (1,\infty)$ and from $-$ to $+$ if $c\in(1/2,1)$. Moreover, for $c<0$ we shall show that the derivative $f'$ is positive on $(0,1)$.

Since $f(0)=f(1)=0$ for $c>0$, this proves that $f\geq 0$ if $c\in(0,1/2)\cup (1,\infty)$ and that $f\leq 0$ if $c\in(1/2,1)$. Moreover, since $f(1)=0$ for $c<0$, this proves that $f\leq 0$ if $c<0$. Thus, we have reduced the proof of Theorem \ref{meansthm} to proving the above sign change properties of $f'$.

In order to discuss the sign changes of $f'$ we compute
\begin{align}\label{der}
f'(t) = \frac{(1-c)(1-t)}{t(1+t)}\left( \frac{1}{\left(\frac{(1+t)^{2}}{4t}\right)^{c}+1} - \frac{t^{c}-t}{(1-c)(t^{c}+1)(1-t)} \right).
\end{align}
Clearly, it suffices to consider the sign changes of the second factor and therefore to consider the sign changes of
\begin{align}\label{sder}
g(t) := \left(\frac{(1+t)^{2}}{4t}\right)^{c}-\left(\frac{(1-c)(t^{c}+1)(1-t)}{t^{c}-t}-1\right).
\end{align}
We shall show that for $c>0$, $g$ has a unique sign change in $(0,1)$ and it changes sign from $-$ to $+$ if $c\in(0,1/2)$ and from $+$ to $-$ if $c\in(1/2,\infty)$. Moreover, for $c<0$ we shall show that $g$ is negative on $(0,1)$. Clearly, these properties of $g$ imply the claimed properties of $f'$ and therefore will conclude the proof.

We next observe that the second term in \eqref{sder} is positive.

\begin{lemma}\label{fraction}
For any $c\in\R\setminus\{1\}$ and $t\in(0,1)$,
$$
\frac{(1-c)(t^{c}+1)(1-t)}{t^{c}-t}>1 \,.
$$
\end{lemma}

\begin{proof}
First, consider the case $c\in[0,1)$. Then concavity of the map $t\mapsto t^{c}$ implies $1-c+ct-t^{c} \geq 0$, therefore $\frac{(1-c)(1-t)}{t^{c}-t}\geq 1$, and the claim follows from $t^{c}+1>1$.

Next, for $c>1$ the argument is similar using convexity of the map $t\mapsto t^c$.

Finally, for $c<0$ convexity of $t \mapsto t^{1-c}$ implies that
\begin{align*}
\frac{(1-c)(t^{c}+1)(1-t)}{t^{c}-t}-1 = \frac{(1-c)(1+t^{-c})(1-t)}{1-t^{1-c}}-1 > \frac{(1-c)(1-t)}{1-t^{1-c}}-1 \geq 0 \,.
\end{align*}
This concludes the proof of the lemma.
\end{proof}

Because of Lemma \ref{fraction}, we can define
\begin{align}\label{lder}
h(t) := c \ln \left(\frac{(1+t)^{2}}{4t}\right) - \ln \left(\frac{(1-c)(t^{c}+1)(1-t)}{t^{c}-t}-1\right).
\end{align}
We shall show that for $c>0$, $h$ has a unique sign change in $(0,1)$ and it changes sign from $-$ to $+$ if $c\in(0,1/2)$ and from $+$ to $-$ if $c\in(1/2,\infty)$. Moreover, for $c<0$ we shall show that $h$ is negative on $(0,1)$. Clearly, these properties of $h$ imply the claimed properties of $g$ and therefore will conclude the proof.

We will prove this by investigating sign changes of $h'$. Namely, we shall show that for $c>0$, $h'$ has a unique sign change in $(0,1)$ and it changes sign from $+$ to $-$ if $c\in(0,1/2)$ and from $-$ to $+$ if $c\in(1/2,\infty)$. Moreover, for $c<0$ we shall show that $h'$ is positive on $(0,1)$.

Let us show that this implies the claimed properties of $h$. Indeed, an elementary limiting argument shows that
$$
h(0) = 
\begin{cases}
-\infty & \text{if}\ c<0 \,, \\
-2c\ln 2 - \ln(1-c) & \text{if}\  c\in(0,1)\,, \\
+\infty & \text{if}\ c>1 \,.
\end{cases}
$$
and
$$
h(1)=0
\qquad\text{for all}\ c \,.
$$
The function $-2c\ln 2 - \ln(1-c)$ is convex on $(0,1)$ and vanishes at $c=0$ and $c=1/2$. From this we conclude that
$$
h(0)<0 \ \text{if}\ c<1/2 \,,
\quad
h(0)=0\ \text{if}\ c=1/2 \,,
\quad
h(0)>0 \ \text{if}\ c>1/2 \,.
$$
Because of this behavior of $h(0)$ and $h(1)$, the claimed properties of $h'$ imply the claimed properties of $h$.

Therefore in order to complete the proof of Theorem \ref{meansthm} we need to discuss the sign changes of $h'$. We compute
\begin{align*}
h'(t) = \frac{v(t)}{(1+t)(t^{c}-t)^{2}\left(    \frac{(1-c)(t^{c}+1)(1-t)}{t^{c}-t}-1  \right)}
\end{align*}
with
$$
v(t) := t(2c^{2}-1)-t^{2}c^{2}+2c(1-2c)(t^{c}-t^{c+1})+t^{2c}(1-2c^{2})+(t^{1+2c}-1)(1-c)^{2}+t^{2c-1}c^{2} \,.
$$
We shall show that for $c>0$, $v$ has a unique sign change in $(0,1)$ and it changes sign from $+$ to $-$ if $c\in(0,1/2)$ and from $-$ to $+$ if $c\in(1/2,\infty)$. Moreover, for $c<0$ we shall show that $v$ is positive on $(0,1)$.

Since, by Lemma \ref{fraction} the denominator in the above expression for $h'$ is positive, these properties of $v$ clearly imply those of $h'$ and therefore complete the proof of the theorem.

In order to prove the claimed properties of $v$ we shall study the sign changes of $v''$. We shall show that for $c>0$, $v''$ has a unique sign change in $(0,1)$ and it changes sign from $+$ to $-$ if $c\in(0,1/2)$ and from $-$ to $+$ if $c\in(1/2,\infty)$. Moreover, for $c<0$ we shall show that $v''$ is positive.

Let us now argue that these properties of $v''$ indeed imply the claimed properties of $v$. We compute
\begin{align}
\label{eq:v2}
v'(t) & =  2c^2-1 - 2c^2 t + 2c(1-2c)( c t^{c-1} - (c+1)t^{c}) + 2c(1-2c^2)t^{2c-1} \notag \\
& \quad\quad\quad+ (1-c)^2(1+2c) t^{2c} + c^2 (2c-1)t^{2c-2},
\notag \\
v''(t) & = 2c\cdot [-c+c(1-2c)(c-1)t^{c-2}-c(1-2c)(c+1)t^{c-1}+(2c-1)(1-2c^{2})t^{2c-2} \notag\\
& \quad \quad \quad +(1-c)^{2}(1+2c)t^{2c-1}+c(2c-1)(c-1)t^{2c-3}],
\end{align}
and finally
\begin{equation}
\label{eq:v3}
v'''(t) = 2c(1-2c)(c-1)t^{c-3} w(t)
\end{equation}
with
$$
w(t) := c(c-2)-(c+1)ct  -2t^{c}(1-2c^{2})+(1-c)(1+2c)t^{c+1}-c(2c-3) t^{c-1} \,.
$$
From these formulas we easily infer that
$$
v(1) = v'(1)=v''(1) = 0 \,,
\qquad
v'''(1) = 2c(1-2c)(c-1)^2 \,.
$$
In particular, $v'''(1)>0$ if $c\in(0,1/2)$ and $v'''(1)<0$ if $c\in(-\infty,0)\cup(1/2,1)\cup(1,\infty)$. This means that $v$ is convex near $t=1$ if $c\in(-\infty,0)\cup(1/2,\infty)$ and concave near $t=1$ if $c\in(0,1/2)$.

Let us discuss the behavior near $t=0$. If $c<1/2$, then $v(t)$ behaves like $t^{2c-1}c^{2}$, so $v(0)=+\infty$, and  $v''(0)>0$. If $c>1/2$, then $v(0)=-(1-c)^{2}$ and $v''(0)<0$.

This behavior of $v$ near $0$ and $1$, together with the claimed sign change properties of $v''$, imply the claimed sign change properties of $v$ and will therefore complete the proof of Theorem \ref{meansthm}.  This is because, for example, if $v$ is convex near $t=1$ with $v(1) = v'(1) = 0$, and $v$ has a single inflection point $t_0 \in (0,1)$, then $v$ is positive on $[t_0,0)$, and $v$ is concave on $(0,t_0)$. 

Thus, we are left with studying the sign changes of $v''$. In order to do so, we need to distinguish several cases. For $c<1$ we will argue via the sign changes of $v'''$, while for $c>1$ we will argue directly.

\bigskip

\emph{Case $c\in(0,1)$.} We want to show that $v''$ changes sign from $+$ to $-$ if $c\in(0,1/2)$ and from $-$ to $+$ if $c\in(1/2,1)$.

Since $v''(0)>0$ if $c\in(0,1/2)$, $v''(0)<0$ if $c\in(1/2,1)$, $v''(1)=0$,  and $v'''(1)>0$, it suffices to show that $v'''$ changes sign only once on $(0,1)$. Because of \eqref{eq:v3} this is the same as showing that $w$ changes sign only once on $(0,1)$. Notice that $w(0)=+\infty$, and  $w(1)=c-1<0$. Moreover,
\begin{align*}
w''(t)=c(1-c)t^{c-3} p(t)
\end{align*}
with
$$
p(t) := t^{2}(c+1)(1+2c)+2t(1-2c^{2})+2c^{2}-7c+6 \,.
$$
The quadratic polynomial $p$ is positive. Indeed, when $c\in(0,1/2)$ this follows from the fact that all its coefficients are positive. When $c\in(1/2,1)$ we observe that the parabola $p$ is minimized on $\R$ at $t=\frac{2c^{2}-1}{(c+1)(1+2c)}$, and its minimal value is  $\frac{(5-3c^{2})+c(11-8c^{2})}{(c+1)(1+2c)}$, which is positive for $c \in (1/2, 1)$. 

The fact that $p$ is positive means that $w$ is convex. Since $w(0)=+\infty$ and $w(1)<0$, we conclude that $w$ has only one root. 

\bigskip

\emph{Case $c\in(-\infty,0)$.} We want to show that $v''$ is positive.

Since $v''(1)=0$, it suffices to show that $v'''$ is negative which, by \eqref{eq:v3}, is the same as showing that $w$ is negative. Clearly, $w(0)=-\infty$, $w''(0)<0$, and $w(1)=c-1<0$, and $w'(1)=(3c-1)(c-1)>0$, so it suffices to show that $w''<0$ on $(0,1)$. For this it suffices to show that $p>0$ on $(0,1)$. We have $p(0)>0$, and $p(1)=9-4c>0$. Thus if $(1+c)(1+2c)\leq 0$ we have proved the claim. Consider the case when $(1+c)(1+2c)>0$. The vertex of the parabola is $t_{0}= \frac{2c^{2}-1}{(c+1)(1+2c)}$. If $c<-1$ then clearly $\frac{2c^{2}-1}{(c+1)(1+2c)}>1$. If $c \in (-1/2, 0)$, then clearly $\frac{2c^{2}-1}{(c+1)(1+2c)}<0$.

\bigskip

\emph{Case $c\in(1,\infty)$.} We want to show that $v''$ changes sign from $-$ to $+$.

We begin with the case $c\in(1,2)$. We write \eqref{eq:v2} as $v''(t)=2c q(t)$ with
\begin{align*}
q(t):=-c+c(1-2c)(c-1)t^{c-2}-c(1-2c)(c+1)t^{c-1}+(2c-1)(1-2c^{2})t^{2c-2}\\
+(1-c)^{2}(1+2c)t^{2c-1}+c(2c-1)(c-1)t^{2c-3}.
\end{align*}
Clearly $q(0)=-\infty$ and $q(1)=0$. It is enough to  show that $q'$ changes sign from  $+$ to $-$.  We have
$$
q'(t)=  t^{2c-4}(2c-1)(c-1) m(t)
$$
with
$$
 m(t) := c(2-c)t^{1-c}+c(c+1)t^{2-c}+2(1-2c^{2})t+(c-1)(1+2c)t^{2}+c(2c-3)\,.
$$
We shall show that $m(t)$  changes sign only once from $+$ to $-$. Clearly $m(0)=+\infty$ and $m''(0)>0$. Next, $m(1)=1-c<0$, and  $m''(1)=(c-1)(c^{2}+2c+2)>0$. Thus it suffices to show $m''>0$ on $(0,1)$. Since $m''(0)>0, m''(1)>0$,  then $m''>0$  will follow from $m'''$ having the constant sign. We have
\begin{align*}
m'''(t) = t^{-c-2} c^{2}(c-1)(c-2)(c+1)(1-t) <0.
\end{align*}
This finishes the case $c \in (1,2)$. 

If $c=2$, then $q(t) = (t-1)(5t^{2}-16t+8)$, and we see that it changes sign only once. 

In what follows we assume $c>2$. Let us rewrite \eqref{eq:v2} as $v''(t) = 2ct^{2c-3} u(t)$ with
\begin{align*}
u(t) & := -ct^{3-2c}+c(1-2c)(c-1)t^{1-c}-c(1-2c)(c+1)t^{2-c}+(2c-1)(1-2c^{2})t \\
& \quad\ +(1-c)^{2}(1+2c)t^{2}+c(2c-1)(c-1) \,.
\end{align*}
We need to show that $u$ changes sign only once.  We have $u(0) = -\infty$, and $u''(0)<0$. At the point $t=1$, we have $u(1)=0$, $u'(1)=-(2c-1)(c-1)^{2}<0$, $u''(1)=-2(2c-1)(c-1)^{2}<0$. It suffices to show that $u''<0$ on $(0,1)$. Since $u''(0)<0, u''(1)<0$, the latter claim will follow from showing that $u'''$ has a constant sign. We have 
\begin{align*}
u'''(t) = t^{-2-c}c(2c-1)(c-1) b(t)
\end{align*}
with
$$
b(t) := c^{3}-c-tc(c+1)(c-2)+2t^{2-c}(2c-3) \,.
$$
The factor $b$ has the property that $b(0)=+\infty$, $b(1)=(c+6)(c-1)>0$. On the other hand,
$$
b'(t) = -c(c+1)(c-2)t^{1-c}\left( t^{c-1}+\frac{2(2c-3)}{c(c+1)}\right)
$$
is negative, so $b$ is positive.

\bigskip

This concludes the proof of the inequality of Theorem \ref{meansthm}.

\subsection{Sharpness of the exponent $2/p$}

The sharpnes of the exponent $2/p$ is easily checked using the variables introduced in Remark~\ref{alter}. If one rpelaces the power of $2/p$ in \eqref{constantA} and kames the 
transforations described there, one is led to the function
 \begin{equation}\label{critq}
 g_{r,p}(s) :=    \eta^{\frac{1}{p-1}}(s)\left( 1 + \left(\frac{1-s}{\eta^{\frac2p}(s)}\right)^r\right) -2\ .
 \end{equation}
 instead of $f_p(s)$. 
A motivation for this reparametrization is that for fixed $p$, the function on the right hand side of  \eqref{constantA} is equal to $1$ up to order ${\mathcal O}((\alpha -1/2)^4)$ at $\alpha=1/2$. In the variable $s$, the leading term in Taylor expansion in $s$ will be second order, and we proves the sharpness by an expansion at this point.

\begin{proof}[Proof of the second paragraph of Theorem~\ref{meansthm}]
 For fixed $r>0$, define the function $ g_{r,p}(s)$ by 
 \eqref{critq}.  By the arithmetic-geometric mean inequality, $(1-s)^{p/2} \leq \eta(s)$ for all $p$, and hence
 $(1-s)/\eta^{2/p}(s) <1$ for $p> 0$, while  $(1-s)/\eta^{2/p}(s) >1$ for $p < 0$. Therefore, for fixed $s$ and $p$, $g_{r,p}(s)$ decreases as $r$ increases for $p > 0$, and does the opposite for $p<0$. 
 
 A Taylor expansion shows that
 $$g_{r,p}(s) =  p(1-r) s + o(s)\ .$$
 It follows that $g_{r,p}(s) \geq 0$ on $[0,1]$ is false (near $s=0$) for $p>2$ and $r> 1$, and for $p < 0$ and $r < 1$. Likewise, it follows that $g_{r,p}(s) \leq 0$ on $[0,1]$ is false  for $p\in (0,2)$ and $r<1$.
 Since the exponent $q$ in \eqref{constantB} corresponds to $r(2/p)$, this together with the reamrks leading to \eqref{toshow}  justifies the statements referring to $q$ in Theorem~\ref{meansthm}. 
 
Consideration of the argument shows that for $p>0$, $p\neq 1,2$, there is equality if and only if
 $\alpha\in \{0, 1/2,1\}$ and for $p < 0$, if and only if $\alpha =1/2$.
 \end{proof}

The proof of Theorem~\ref{meansthm} is now complete.   By what has been explained above, 
the inequality of Theorem~\ref{main} is proved. Concerning the cases of equality, we have seen in Section 3, that for all values of $p$ under consideration, if there is equality then 
$\max\{\alpha(a), 1-\alpha(x)\}$ is constant, and then by what has just been proved here, and in Lemma~\ref{constb} for $p>0$,
this constant is either $1$, in which case $f$ and $g$ have essentially disjoint support, or $1/2$ in which case $ = g$. For $p < 0$, there is equality only in case $f=g$. Finally, it is evident that there is equality in these cases.


\section{Doubling arguments and a generalization to Schatten norms.}

\subsection{Doubling arguments}\label{sdoubling}

We begin this section with a simple proof showing  that if 
the inequality \eqref{maineq} is valid for some  $p\geq 2$ or some $p<0$, then it is also valid for 
$2$. Since the inequality \eqref{maineq} holds as an identity for $p=2$, and is simple to prove for $p= -1$ (see Remark~\ref{alter}), this yields a simple proof of infinitely many cases of the inequality \eqref{maineq}.  The proof is not only simple and elegant; it applies to certain non-commutative generalizations of \eqref{maineq} for which the reductions in parts A and B of the proof we have just presented are not applicable, 
as we discuss.

To introduce the doubling argument
we present a direct proof of Theorem \ref{main} for $p=4$.

\begin{proof}[Direct proof of Theorem~\ref{main} for $p=4$]  Suppose $f,g\geq 0$, By homogeneity, we may suppose that
$\|f\|_4^4 + \|g\|_4^4 =2$. Define 
\begin{equation}\label{XYdef}
X := fg\ , \quad Y := f^2+g^2\ , \quad \alpha := \|X\|_2 \quad{\rm and}\quad \beta := \|Y\|_2\ .
\end{equation}
By the arithmetic-geometric mean inequality, $X \leq \frac12 Y$, and hence
$$\int X^2{\rm d}\mu \leq \frac14 \int Y^2{\rm d}\mu = \frac14 \int (f^4+ g^4 + 2f^2g^2){\rm d}\mu = \frac12 + \frac 12 
\int X^2{\rm d}\mu\ .$$
This yields $\alpha \leq 1$ and $\beta \leq 2$. 
Then $(f+g)^2 = Y + 2X$ and hence
\begin{equation}\label{p4A}
\|f+g\|_4^2 = \|Y+2X\|_2 \leq \|Y\|_2 + 2\|X\|_2 = \beta +2\alpha\ .
\end{equation}
It suffices to prove that 
$\beta +2\alpha  \leq 2^{1/2} (1 +\alpha)^{3/2}$.
Note that $\beta^2 = \int (f^2+ g^2)^2{\rm d}\mu = 2 + 2\alpha^2$, and then since $\alpha \in [0,1]$.
Thus it suffices to show that 
\begin{equation}\label{p4B}
(1+\alpha^2)^{1/2}    \leq  (1 +\alpha)^{3/2} - 2^{1/2}\alpha \quad{\rm for\ all }\quad  0 \leq \alpha \leq 1\ .
\end{equation}
Squaring both sides, this is equivalent to $1+\alpha^2 \leq (1+\alpha)^3 + 2\alpha^2 - 2^{3/2}\alpha(1+\alpha)^{3/2}$. This reduces to 
$2^{3/2}(1+\alpha)^{3/2} \leq 3 + 4\alpha + \alpha^2$.
Squaring both sides again, this reduces to $(\alpha^2 -1)^2 \geq 0$, completing the proof. 
\end{proof} 

What made this proof work is the fact that the inequality holds for $p=2$ -- as an identity, but that is unimportant. Then, using Minkowski's inequality, as in \eqref{p4A}, together with the numerical inequality \eqref{p4B} we arrive at the inequality for $p=4$.  This is a first instance of the 
general doubling 
proposition, to be proved next.  The inequality \eqref{p4B} is s special case of the general inequality \eqref{p8C2} proved below.

This strategy can be adapted to give direct proof of the inequality for other integer values of $p$; e.g., $p=3$.
When $p$ is an integer, and $f$ and $g$ are non-negative, one has the binomial expansion of $(f+g)^p = f^p + g^p + {\rm mixed\ terms}$.  Under the assumption that $\int (f^p + g^p) =2$, one is left with estimating the mixed terms, and one can use H\"older for this.  When $p$ is not an integer, there is no useful expression for $(f+g)^p - f^p - g^p$.

\begin{proposition}[A ``doubling'' argument]  \label{double} Suppose that 
for some $p\geq 2$, 
\eqref{maineq} is valid for all $f,g\geq 0$.
Then \eqref{maineq} is valid with $p$ replaced by $2p$ for all $f,g\geq 0$.  Likewise, if for some $p < 0$ the reverse of 
\eqref{maineq} is valid for all $f,g > 0$, then  the reverse of \eqref{maineq} is valid with $p$ replaced by $2p$ for all $f,g> 0$.
\end{proposition}

The proof of Proposition~\ref{double} relies on the following lemma.

\begin{lemma}\label{doubleB}
For $t\in \R$, define   $\psi_t$ on $[0,\infty)$ by
\begin{equation}\label{p8C2}
\psi_t(\alpha) = (1 + \alpha)^{1+t} -  (1 + \alpha^2)^{t} - 2^{t}\alpha\ .
\end{equation}
Then for $t\in [0,1]$, $\psi_t(\alpha) \geq 0$ on $[0,\infty)$, while for $t>1$, $\psi_t(\alpha) \leq 0$ on $[0,\infty)$.
\end{lemma}

\begin{proof}
We write $\psi_t(\alpha)=(1+\alpha)^t - (1+\alpha^2)^t - (2^t - (1+\alpha)^t)\alpha$. Therefore,
$$\frac{\psi_t(\alpha)}{\alpha(1-\alpha)} = \frac{(1+\alpha)^t - (1+\alpha^2)^t}{\alpha(1-\alpha)} - \frac{2^t - (1+\alpha)^t}{1-\alpha}\ .$$
Defining $a := 1+\alpha^2$, $b := 1 +\alpha$ and $c :=2$, and defining $\varphi(\alpha) := x^t$, the right hand side is the same as
$$\frac{\varphi(b) - \varphi(a)}{b-a} - \frac{\varphi(c) - \varphi(b)}{c-b}\ .$$
For $\alpha\in[0,1)$ we have $a<b<c$ and therefore this quantity is positive when $\varphi$ is concave, and negative when $\varphi$ is convex. For $\alpha\in(1,\infty)$ we have $a>b>c$ and therefore this quantity is negative when $\varphi$ is concave, and positive when $\varphi$ is convex.
\end{proof}

\begin{proof}[Proof of Proposition~\ref{double}]
Let $f,g\in L^{2p}$ with $\|f\|_{2p}^{2p} + \|g\|_{2p}^{2p} =2$. 
Define $X := fg$ and $Y := f^2+g^2$, and $\gamma := \|X\|_p$ and $\beta := \|Y\|_p$. 
By the triangle inequality we have
$$
\|f+g\|_{2p}^2 = \|Y+2X\|_p
\begin{cases}
\leq \|Y\|_p+2\|X\|_p = \beta + 2\gamma & \text{if}\ p\geq 2 \,,\\
\geq \|Y\|_p+2\|X\|_p = \beta + 2\gamma & \text{if}\ p<0 \,.
\end{cases}
$$
(Note that the triangle inequality reverses for $p<0$.) We now use the assumption that the inequality \eqref{maineq} is valid for $p$. Applying the inequality with exponent $p$ to the functions $f^2$ and $g^2$, which satisfy $\| f^2 \|_p^p + \|g^2\|_p^p = \|f\|_{2p}^{2p} + \|g\|_{2p}^{2p} =2$, we obtain for $p\geq 2$,
$$
\beta^p = \| f^2 + g^2 \|_p^p \leq 2 \left( 1+ \|f^2 g^2\|_{p/2}\right)^{p-1} = 2\left(1+\gamma^2\right)^{p-1}
$$
and similarly $\beta^p \geq 2\left(1+\gamma^2\right)^{p-1}$ for $p<0$. To summarize, we have shown that
$$
\|f+g\|_{2p}^2
\begin{cases}
\leq 2^{1/p} (1+\gamma^2)^{1-1/p} + 2\gamma & \text{if}\ p\geq 2 \,,\\
\geq 2^{1/p} (1+\gamma^2)^{1-1/p} + 2\gamma & \text{if}\ p<0 \,.
\end{cases}
$$
According to Lemma \ref{doubleB} (with $t=1-1/p$ and $\alpha=\gamma$) this is bounded from above for $p\geq 2$ and from below for $p<0$ by $2^{1/p} (1+\gamma)^{2-1/p} = 2^{1/p} (1+ \|fg\|_p)^{2-1/p}$, which is the claimed inequality.
\end{proof}

\subsection{A generalization to Schatten norms}

For $p\in [1,\infty)$, an operator $A$ on some Hilbert space belongs the Schatten $p$-class ${\mathcal S}_p$ in case $(A^*A)^{p/2}$ is trace class, and the Schatten $p$ norm on
${\mathcal S}_p$ is defined by $\|A\|_p = (\tr[(A^*A)^{p/2}])^{1/p}$.   
One possible non-commutative analog of  (part of) Theorem~\ref{main} would assert that for positive $A,B\in {\mathcal S}_p$, $p > 2$.
\begin{equation}\label{Schatv}
\tr(A+B)^p  \leq \left(1+ \left(
\frac{
\tr[B^{p/4}A^{p/2}B^{p/4}] }
{\tfrac12 \Vert A\Vert_p^p + \tfrac12 \Vert B\Vert_p^p}
\right)^{2/p} 
\right)^{p-1} 
\tr \left(\, A^p + B^p\, \right).
\end{equation}
Note that for $p=2$, \eqref{Schatv} holds as an identity. 

In this setting, it is not clear how to implement analogs of Parts A and B of our proof for functions. However, the direct proofs sketched at the beginning of this section
do allow us to prove the valididty of \eqref{Schatv} for all $p= 2^k$, $k\in \N$. 

\begin{theorem}\label{Schc}  If \eqref{Schatv} is valid for some $p\geq 2$ and all positive $A,B\in {\mathcal S}_p$, then it is valid for $2p$ and all
 $A,B\in {\mathcal S}_{2p}$. In particular, since \eqref{Schatv} holds as an identity for $p=2$, it is valid for $p=2^k$ for all $k\in\N$. 
\end{theorem}

\begin{proof}
Let $A$ and $B$ be positive operators in ${\mathcal S}_{2p}$, and assume that $\|A\|_{2p}^{2p} + \|B\|_{2p}^{2p} =2$, which, by homogeneity, 
entails no loss of generality.    Define
$$X := \frac12(AB + BA) \qquad{\rm and}\qquad  Y = A^2+B^2\ .$$
Note that 
$$\|X\|_p \leq \frac12( \|AB\|_p + \|BA\|_p)\ .$$
By definition, the Lieb--Thirring inequality \cite{LT}, and cyclicity of the trace,
$$\|AB\|_p^p = \tr[(BA^2B)^{p/2}] \leq \tr[B^{p/2}A^pB^{p/2}] =  \tr[A^{p/2}B^pA^{p/2}]\ .$$
Define 
$$\beta := \|Y\|_p\quad{\rm  and}\quad  \gamma := (\tr[B^{p/2}A^pB^{p/2}])^{1/p}\ .$$
Therefore,
$\|A+B\|_{2p}^2 = \|Y + 2X\|_p \leq \|Y\|_p + 2 \|X\|_p \leq   \beta + 2\gamma$. 
Since $\|A^2\|_p^p + \|B^2\|_p^p =2$, we can apply \eqref{Schatv} to deduce that
$$
\beta^p = \| A^2 + B^2 \|_p^p \leq 2 \left( 1+ (\tr[B^{2p/4}A^{p}B^{2p/4}])^{2/p}\right)^{p-1} = 2\left(1+\gamma^2\right)^{p-1}.
$$
Altogether 
$$\|A+B\|_{2p}^2  \leq 2^{1/p} (1+\gamma^2)^{1-1/p} + 2\gamma $$
and, by Lemma \ref{doubleB}, the right side is bounded above by $2^{1/p} (1+\gamma)^{2-1/p}$, which proves the inequality, 
\end{proof}

\noindent{\bf Acknowledgement} We thank Anthony Carbery for useful correspondence.


\bibliographystyle{amsalpha}

\end{document}